\documentclass[final]{siamart171218}

\pdfoutput=1

\usepackage{bm}
\usepackage{amsmath,amssymb}
\usepackage{epsfig,graphics}
\usepackage{subfigure,sidecap}
\usepackage{wrapfig, comment}
\usepackage{colortbl,hhline}
\usepackage{mathtools}
\usepackage{paralist}

\usepackage[norefs,nocites]{refcheck}
\usepackage{autonum}

\usepackage[textsize=footnotesize]{todonotes}

\newcommand{\eps}{\varepsilon}

\def\x{{\bf x}}
\def\y{{\bf y}}

\def\beq{\begin{equation}}
\def\eeq{\end{equation}}

\def\eps{\epsilon}

\def\rprime{\hbox{\hskip.25em\raise.5ex\hbox{$'$}\hskip.15em}}

\def\ddn1{{\frac{\partial}{\partial \nu_{\yb}}}}

\newcommand{\Lcal}{\mathcal{L}}
\newcommand{\Asub}{_{\text{\scriptsize A}}}
\newcommand{\shl}{\hspace*{-0.1em}<\hspace*{-0.1em}}

\DeclareMathOperator*{\argmin}{arg\;min}

\usepackage{algorithmic}

\theoremstyle{remark}

\theoremstyle{plain}
\newtheorem{thm}{Theorem}
\newtheorem{prop}[thm]{Proposition}

\newlength{\kaka}
\newcommand{\dO}{\partial\OO}
\newcommand{\G}{\Gamma}
\newcommand{\sheq}{\hspace*{-0.1em}=\hspace*{-0.1em}}
\newcommand{\shcup}{\hspace*{-0.1em}\cup\hspace*{-0.1em}}
\newcommand{\Div}{\mbox{div}\,} 
\newcommand{\bfu} {\boldsymbol{u}}
\newcommand{\bfsig}{\boldsymbol{\sigma}}
\newcommand{\bfze}{\mathbf{0}}
\newcommand{\OO}{\Omega}
\newcommand{\bfI} {\boldsymbol{I}}
\newcommand{\bfD} {\boldsymbol{D}}
\newcommand{\bfuD} {\bfu\Dsup}
\newcommand{\bfna}{\boldsymbol{\nabla}}
\newcommand{\shp}{\hspace*{-0.1em}+\hspace*{-0.1em}}
\newcommand{\Tsup}{^{\text{\scriptsize T}}}
\newcommand{\Dsup}{^{\text{\scriptsize D}}}
\newcommand{\bfn} {\boldsymbol{n}}
\newcommand{\bff} {\boldsymbol{f}}
\newcommand{\sip} {\! \cdot\!}
\newcommand{\shin}{\hspace*{-0.1em}\in\hspace*{-0.1em}}
\newcommand{\bftau}{\boldsymbol{\tau}}
\newcommand{\VS}{\text{\boldmath $\Vcal$}}
\newcommand{\Pcal}{\mathcal{P}}
\newcommand{\bfv} {\boldsymbol{v}}
\newcommand{\FS}{\ensuremath{\mbox{\boldmath $\mathcal{F}$}}}
\newcommand{\lbra}{\big\langle\hspace*{0.1em}}
\newcommand{\rbra}{\hspace*{0.1em}\big\rangle}
\newcommand{\Vcal}{\mathcal{V}}
\newcommand{\bfg} {\boldsymbol{g}}
\newcommand{\iO}{\int_{\OO}}
\newcommand{\dip} {\! :\!}
\newcommand{\dV}{\;\text{d}V}

\newcommand{\tdemi} {\tfrac{1}{2}}
\newcommand{\bfx} {\boldsymbol{x}}

\newcommand{\shm}{\hspace*{-0.1em}-\hspace*{-0.1em}}
\newcommand{\PLsub}{_{\text{\scriptsize PL}}}
\newcommand{\bfe} {\boldsymbol{e}}
\newcommand{\iG} {\int_{\G}}
\newcommand{\Vsub}{_{\text{\scriptsize V}}}
\newcommand{\inv}[1]{\dfrac{1}{#1}}
\newcommand{\Qsub}{_{\text{\scriptsize Q}}}
\newcommand{\bfth}{\boldsymbol{\theta}}
\newcommand{\Rbb} {\mathbb{R}}

\newcommand{\bfa} {\boldsymbol{a}}

\newcommand{\Ocal}{\mathcal{O}}
\newcommand{\Gpp}{\G_0}
\newcommand{\Gpm}{\G_L}

\newcommand{\Gp}{\G_{p}}
\newcommand{\dd}[1]{\protect\settowidth{\kaka}{\mbox{$#1$}}\protect\makebox[\kaka]{$\buildrel{\scriptscriptstyle\star}\over{#1}$}}

\newcommand{\bfp} {\boldsymbol{p}}
\newcommand{\lsqb}{\big[\hspace*{0.1em}}
\newcommand{\rsqb}{\hspace*{0.1em}\big]}
\newcommand{\lpar}{\big(\hspace*{0.1em}}
\newcommand{\rpar}{\hspace*{0.1em}\big)}
\newcommand{\shsubs}{\hspace*{-0.1em}\subset\hspace*{-0.1em}}
\newcommand{\der}[2]{\dfrac{\text{d}#1}{\text{d}#2}}
\newcommand{\shdeq}{\hspace*{-0.1em}:=\hspace*{-0.1em}}
\newcommand{\idO}{\int_{\dO}}
\newcommand{\Ssub}{_{\text{\scriptsize S}}}
\newcommand{\iS}{\int_{S}}
\newcommand{\DivS} {\mbox{div}_{\! S}}
\newcommand{\tens}{\hspace*{-1pt}\otimes\hspace*{-1pt}}

\newcommand{\bfxi}{\boldsymbol{\xi}}
\newcommand{\hatu}{\hat{u}}
\newcommand{\shtimes}{\hspace*{-0.1em}\times\hspace*{-0.1em}}

\newcommand{\shg}{\hspace*{-0.1em}>\hspace*{-0.1em}}
\newcommand{\bfz} {\boldsymbol{z}}
\newcommand{\shleq}{\hspace*{-0.1em}\leq\hspace*{-0.1em}}

\newcommand{\Lcb}{\Big\{\hspace*{0.1em}}
\newcommand{\Rcb}{\hspace*{0.1em}\Big\}}
\newcommand{\bfuh} {\hat{\boldsymbol{u}}{}}
\newcommand{\bffh} {\hat{\boldsymbol{f}}{}}
\newcommand{\hatp}{\hat{p}}
\newcommand{\bfA} {\boldsymbol{A}}

\newcommand{\nS}{\bfna_{\!\! S}}
\newcommand{\Lsqb}{\Big[\hspace*{0.1em}}
\newcommand{\Rsqb}{\hspace*{0.1em}\Big]}

\newcommand{\Lpar}{\Big(\,}
\newcommand{\Rpar}{\,\Big)}

\newcommand{\dx}{\,\text{d}x}
\newcommand{\fhat}{\hat{f}}
\newcommand{\phat}{\hat{p}}
\newcommand{\suite}[1][0ex]{\notag \\[#1] & \mbox{}\hspace{15pt}}
\newcommand{\bfU} {\boldsymbol{U}}

\newcommand{\ds}{\,\text{d}s}

\newtheorem{remark}{Remark}

\begin{document}             
\title{Shape optimization of Stokesian peristaltic pumps using boundary integral methods}

\author{Marc Bonnet\thanks{POEMS (CNRS, INRIA, ENSTA), ENSTA, 91120 Palaiseau, France. mbonnet@ensta.fr}
\and Ruowen Liu\thanks{Department of Mathematics, University of Michigan, Ann Arbor, United States. ruowen@umich.edu}
\and Shravan Veerapaneni\thanks{Department of Mathematics, University of Michigan, Ann Arbor, United States. shravan@umich.edu}
}

\date{\today}

\maketitle
\begin{abstract} 
This article presents a new boundary integral approach for finding optimal shapes of peristaltic pumps that transport a viscous fluid. Formulas for computing the shape derivatives of the standard cost functionals and constraints are derived. They involve evaluating physical variables (traction, pressure, etc.) on the boundary only. By employing these formulas in conjunction with a boundary integral approach for solving forward and adjoint problems, we completely avoid the issue of volume remeshing when updating the pump shape as the optimization proceeds. This leads to significant cost savings and we demonstrate the performance on several numerical examples.  
\end{abstract} 


\maketitle

\section{Introduction}
Many physiological flows are realized owing to {\em peristalsis}, a transport mechanism induced by periodic contraction waves in fluid-filled (or otherwise) tubes/vessels \cite{jaffrin1971peristaltic, fauci2006biofluidmechanics, bornhorst2017gastric}. This mechanism is used in engineering applications like microfluidics, organ-on-a-chip devices and MEMS devices for transporting and mixing viscous fluids at small scale (e.g., see \cite{teymoori2005design, wang2006pneumatically, skafte2009multi, kim2012human, zhang2015valve, esch2015organs}). Owing to its importance in science and technological applications, numerous analytical and numerical studies were carried out in the past decades to characterize the fluid dynamics of peristalsis in various physical scenarios; some recent works include \cite{teran2008peristaltic, tripathi2014mathematical, khabazi2016peristaltic} for non-Newtonian flows and \cite{chrispell2011peristaltic, aranda2015model, khabazi2016peristaltic2} for particle transport.


Several applications---optimal transport of drug particles in blood flow \cite{mekheimer2018peristaltic}, understanding sperm motility in the reproductive tract \cite{simons2018sperm}, propulsion of soft micro-swimmers \cite{farutin2013amoeboid,shi2016study}---require scalable numerical methods that can handle arbitrary shape deformations. One of the challenges of existing mesh-based methods (e.g., finite element methods) is the high computational expense of re-meshing, needed to proceed between optimization updates or in transient solution of the forward/adjoint problems. Boundary integral equation (BIE) methods, on the other hand, avoid volume discretization altogether for linear partial differential equations (PDEs) and are highly scalable even for moving geometry problems \cite{petascale}. While BIE methods have been used widely for shape optimization problems, including in linear elasticity, acoustics, electrostatics, electromagnetics and heat flow (e.g., \cite{bordas:16, zheng:12, bandara2015boundary, lelouer:16, harbrecht2013numerical}),
 we are not aware of their application to optimization of peristaltic pumps transporting simple (or complex) fluids.   

The primary goal of this work is to derive shape sensitivity formulas that can be used in conjunction with an indirect BIE method to enable fast numerical optimization routines. Our work is inspired by that of Walker and Shelley \cite{walk:shel:10} who considered the shape optimization of a peristaltic pump transporting a Newtonian fluid at low to moderate Reynolds numbers ($Re$) and applied a finite element discretization. Here, we restrict our attention to problems in the zero $Re$ limit only. Following \cite{walk:shel:10}, we consider shapes that minimize the input fluid power under constant volume and flow rate constraints. The forward problem requires solving the Stokes equations in a tube with periodic flow conditions and prescribed slip on the walls while the adjoint problem has a prescribed pressure drop condition across the tube. For both problems, we employ the recently developed periodic BIE solver of Marple et al. \cite{barnett:16}; the primary advantage compared to classical BIE methods that rely on periodic Green's functions is that the required slip or pressure-drop conditions can be applied directly.

The paper is organized as follows. In Section \ref{sc:formulation}, we define the shape optimization problem and the PDE formulation of the forward problem in strong and weak forms. In Section \ref{sc:sensitivities}, we derive the shape sensitivity formulas for a specific objective function and the functionals required to impose the given constraints on the pump shape. Using these formulas, we present a numerical optimization procedure in Section \ref{sc:numerical} based on a periodic boundary integral equation formulation for the PDE solves. We present validation and shape optimization results in Section \ref{sc:results} followed by conclusions in Section \ref{sc:conclusions}.

\section{Problem formulation} \label{sc:formulation}
\subsection{Formulation of the wall motion}

Pumping is achieved by the channel wall shape moving along the positive direction $\bfe_1$ at a constant velocity $c$, as a traveling wave of wavelength $L$ (the wave period therefore being $T=L/c$). The quantities $L,c$ (and hence $T$) are considered as fixed in the wall shape optimization process. This apparent shape motion is achieved by a suitable material motion of the wall, whose material is assumed to be flexible but inextensible. Like in~\cite{walk:shel:10}, it is convenient to introduce a \emph{wave frame} that moves along with the traveling wave, i.e. with velocity $c\bfe_1$ relative to the (fixed) lab frame.\enlargethispage*{1ex}

\begin{figure}[h]
\begin{center}
  \includegraphics[width=0.6\textwidth]{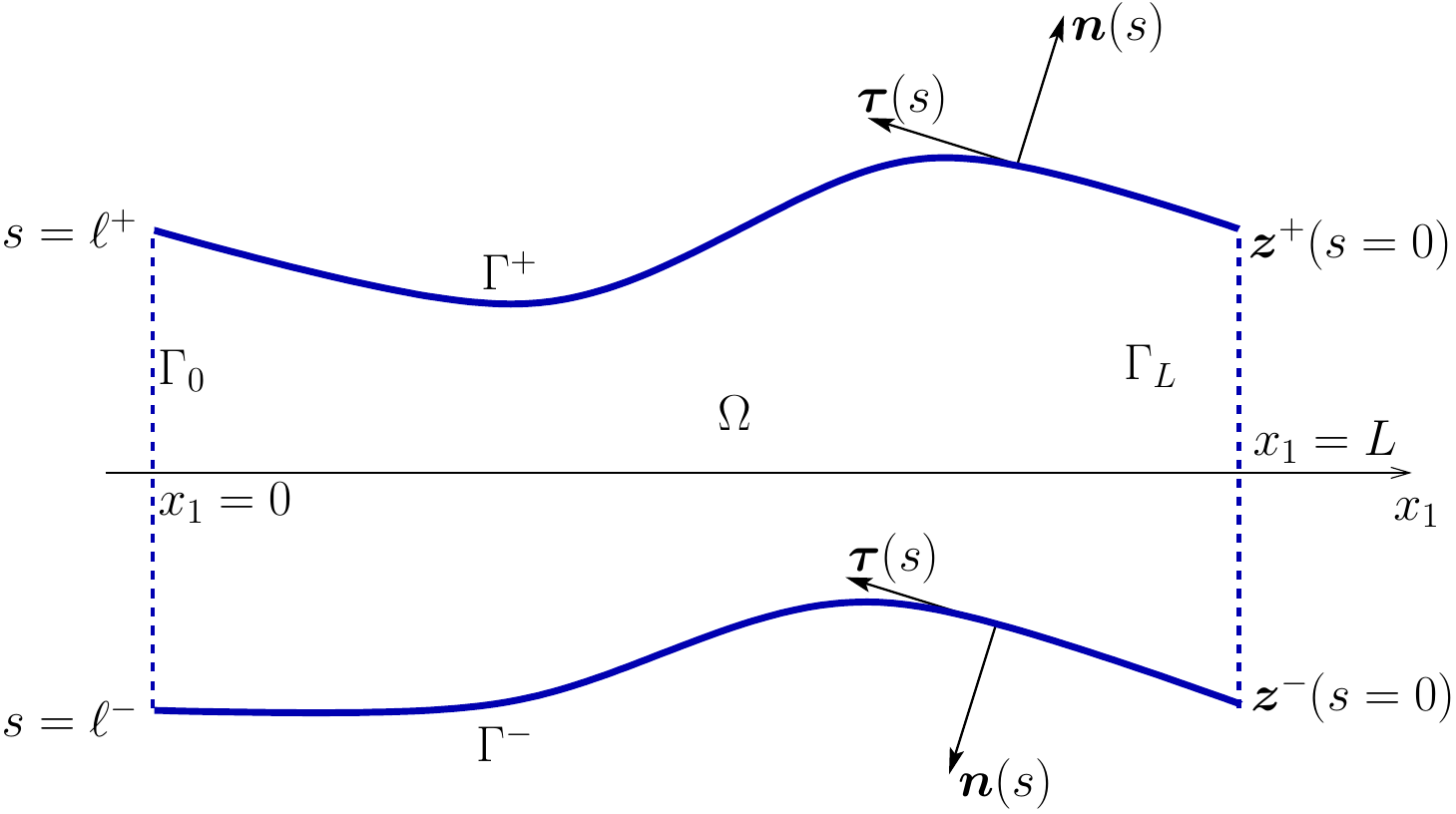}
  \caption{Channel in wave frame, 2D configuration: geometry and notation}\label{geom:2D}
\end{center}
\end{figure} 

Let then $\Omega$ denote, in the wave frame, the fluid region enclosed in one wavelength of the channel (see Fig.~\ref{geom:2D}), whose boundary is $\dO=\G\cup\G_p$. The wall $\G\shdeq\G^{+}\shcup\G^{-}$, which is fixed in this frame, has disconnected components $\G^{\pm}$ which are not required to achieve symmetry with respect to the $x_1$ axis and have respective lengths $\ell^{\pm}$. The  remaining contour $\G_p\shdeq\G_0\shcup\G_L$ consists of the periodic planar end-sections $\G_0$ and $\G_L$, respectively situated at $x_1=0$ and $x_1=L$; the endpoints of $\G_L$ are denoted by $\bfz^{\pm}$ (Fig.~\ref{geom:2D}). Both channel walls are described as arcs $s\mapsto \bfx^{\pm}(s)$ with the arclength coordinate $s$ directed ``leftwards'' as depicted in~Fig.~\ref{geom:2D}, whereas the unit normal $\bfn$ to $\dO$ is everywhere taken as outwards to $\OO$. The position vector $\bfx(s)$, unit tangent $\bftau(s)$, unit normal $\bfn(s)$ and curvature $\kappa(s)$ obey the Frenet formulas
\begin{equation}
  \bfx_{,s} = \bftau, \qquad \bftau_{,s} = \kappa\bfn, \qquad
  \bfn_{,s} = -\kappa\bftau \qquad \text{on $\G^+$ and $\G^-$}. \label{frenet}
\end{equation}
The opposite orientations of $(\bftau,\bfn)$ on $\G^+$ and $\G^-$ resulting from our choice of conventions imply opposite sign conventions on the curvature, which is everywhere on $\G$ taken as $\kappa\shdeq\bftau_{,s}\sip\bfn$ for overall consistency.

In the wall frame, the wall material velocity must be tangent to $\G$ (wall material points being constrained to remain on the surface $\G$); moreover the wall material is assumed to be inextensible. In the wave frame, the wall material velocities $\bfU$ satisfying both requirements must have, on each wall, the form
\begin{equation}
  \bfU(s) = U\bftau(s),
\end{equation}
where $U$ is a constant. Moreover, in the wave frame, all wall material points travel over an entire spatial period during the time interval $T\sheq L/c$, which implies $U=c\ell/L$. Finally, the viscous fluid must obey a no-slip condition on the wall, so that the velocity of fluid particles adjacent to $\bfx(s)$ is $\bfU(s)$. Concluding, the pumping motion of the wall constrains on each wall the fluid motion through
\begin{equation}
  \bfu(\bfx) = \bfuD(\bfx) := \frac{\ell^{\pm}}{L}c\bftau^{\pm}(\bfx) ,\quad \bfx\in\G^{\pm}, \label{noslip:2D}
\end{equation}

\begin{remark}\label{noslip:remark}
The no-slip condition $\bfu\sheq\bftau$ proposed in~\cite{walk:shel:10}, which corresponds with the present notations to setting $U\sheq c$, is inconsistent with the period $T\sheq L/c$ of the traveling wave, unless the channel is straight ($\ell^{\pm}\sheq L$). This discrepancy alters the boundary condition~\eqref{noslip:2D} and hence the shape optimization problem (the value of the power loss functional for given wall shape being affected in wall shape-dependent fashion).
\end{remark}

\subsection{PDE formulation of forward problem}
\label{forward:pde}

The stationary Stokes flow in the wave frame~\cite{walk:shel:10} solves the governing PDE formulation
\begin{equation}
\begin{aligned}
  \text{(a) } && \Div\bfsig[\bfu,p] &= \bfze && \text{in $\OO$} && \text{(balance of momentum)} \\
  \text{(b) } && \bfsig[\bfu,p] &= -p\bfI + 2\mu\bfD[\bfu] && \text{in $\OO$} && \text{(constitutive relation)} \\
  \text{(c) } && \Div\bfu &= 0 && \text{in $\OO$} && \text{(incompressibility)} \\
  \text{(d) } && \bfu &= \bfuD && \text{on $\G$} && \text{(Dirichlet BC on wall)} \\
  \text{(e) } && \bfu &\text{\ periodic} && \text{on $\G_p$}
\end{aligned} \label{forward:PDE}
\end{equation}
where $\bfu$, $p$ are the velocity and pressure, $\bfD[\bfu]:=(\bfna\bfu\shp\bfna\Tsup\bfu)/2$ is the strain rate tensor, and $\bfsig[\bfu,p]$ is the stress tensor. Equations~(\ref{forward:PDE}a,b,c) respectively express the local balance of momentum (absent any body forces), the constitutive relation ($\mu$ being the dynamic viscosity) and the incompressibility condition. The prescribed wall velocity $\bfuD$ is given by~\eqref{noslip:2D}. Problem~\eqref{forward:PDE} defines $p$ up to an arbitrary additive constant.\enlargethispage*{1ex}

\subsection{Weak formulation of forward problem}
\label{Forward:weak}

In this study, flow computations rely on a boundary integral equation formulation, see Section~\ref{sec:forwardsolver}. It is however convenient, for later establishment of shape derivative identities, to recast the forward problem~\eqref{forward:PDE} as a mixed weak formulation (e.g. \cite{brez:fort:91}, Chap. 6):
\begin{equation}
\text{find }\bfu\shin\VS, \ p\shin\Pcal, \ \bff\shin\FS, \qquad\left\{
\begin{aligned}
  \text{(a) \ }&& a(\bfu,\bfv) - b(\bfv,p) - \lbra\bff,\bfv\rbra_{\G} &= 0
  &\qquad& \forall\bfv\shin\VS \\
  \text{(b) \ }&& b(\bfu,q) &= 0 && \forall q\shin\Pcal \\
  \text{(c) \ }&& \lbra\bfuD,\bfg\rbra_{\G} - \lbra\bfu,\bfg\rbra_{\G} &= 0 && \forall\bfg\shin\FS
\end{aligned} \right. \label{forward:weak}
\end{equation}
where $\lbra\cdot,\cdot\rbra_{\G}$ stands for the $L^2(\G)$ duality product, and the bilinear forms $a$ and $b$ are defined by
\begin{equation}
  a(\bfu,\bfv) = \iO 2\mu\bfD[\bfu]\dip\bfD[\bfv] \dV, \qquad b(\bfv,q) = \iO q\,\Div\bfv \dV
\end{equation}
The function spaces in problem~\eqref{forward:weak} are as follows: $\VS$ is the space of all periodic vector fields contained in $H^1(\OO;\Rbb^2)$, $\Pcal$ is the space of all $L^2(\OO)$ functions with zero mean (i.e. obeying the constraint $\lbra p,1 \rbra_{\OO}\sheq0$) and $\FS\sheq H^{-1/2}(\G;\Rbb^2)$. The chosen definition of $\Pcal$ caters for the fact that $p$ would otherwise be defined only up to arbitrary additive constants. The Dirichlet boundary condition~(\ref{forward:PDE}d) is (weakly) enforced explicitly through~(\ref{forward:weak}c), rather than being embedded in the velocity solution space $\VS$, as this will facilitate the derivation of shape derivative identities. The unknown $\bff$, which acts as the Lagrange multiplier associated with the Dirichlet boundary condition~(\ref{forward:PDE}d), is in fact the force density (i.e stress vector) $\bfsig[\bfu,p]\sip\bfn$.\enlargethispage*{1ex}

\subsection{Formulation of optimization problem, objective functional}
\label{optim:pb}

We address the problem of finding the channel shape that makes peristaltic pumping most efficient, i.e., leads to minimum power loss for given mass flow rate. The power loss cost functional~\cite{walk:shel:10} is defined by
\begin{equation}
  J\PLsub(\OO) := \lbra \bff_{\OO},(\bfuD\shp c\bfe_1) \rbra_{\G}. \label{J:loss:def}
\end{equation}
It is a shape functional in that its value is completely determined by the domain $\OO$ (in a partly implicit way through $\bff_{\OO}$). The minimization of $J\PLsub(\OO)$, as in~\cite{walk:shel:10}, subject to two constraints, involvesd two other shape functionals:
\begin{enumerate}[(a)]
\item The channel volume per wavelength has a prescribed value $V_0$:
\begin{equation}
  C\Vsub(\OO) := |\Omega| - V_0 = 0 \label{vol:channel}
\end{equation}
\item The mass flow rate per wavelength $Q(\Omega)$ has a prescribed value $Q_0$. Expressing this constraint in the wave frame gives:
\begin{equation}
  C\Qsub(\OO) := Q(\OO) - Q_0 = 0 \qquad\text{with}\quad
  Q(\OO) = \inv{L}\iO (\bfu_{\OO}\shp c\bfe_1)\sip\bfe_1 \dV. \label{flowrate}
\end{equation}
\end{enumerate}

\begin{remark}\label{JPL}
The power loss functional $J\PLsub(\OO)$ is insensitive to modification of $p$ by an additive constant. This is seen, for instance, by combining~(\ref{forward:weak}a) with $\bfv\sheq\bfu\shp c\bfe_1$ and~(\ref{forward:weak}b) with $q\sheq p$, which yields
\begin{equation}
   J\PLsub(\OO) = \lbra \bff_{\OO},(\bfuD\shp c\bfe_1) \rbra_{\G} = a(\bfu_{\OO},\bfu_{\OO}).
\end{equation}
\end{remark}

\section{Shape sensitivities} \label{sc:sensitivities}
In this section, we review available shape derivative concepts that address our needs. Rigorous expositions of shape sensitivity theory can be found in~\cite[Chaps. 8,9]{del:zol} or \cite[Chap. 5]{henrot:pierre:18}.

\subsection{Shape sensitivity analysis: an overview}

Let $\OO_{\text{all}}\shsubs\Rbb^2$ denote a fixed domain chosen so that $\OO\Subset\OO_{\text{all}}$ always holds for the shape optimization problem of interest. Shape changes are described with the help of \emph{transformation velocity} fields, i.e. vector fields $\bfth:\OO_{\text{all}}\to\Rbb^2$ such that $\bfth=\bfze$ in a neighborhood $\dO_{\text{all}}$. Then, shape perturbations of domains $\OO\Subset\OO_{\text{all}}$ can be mathematically described using a pseudo-time $\eta$ and a geometrical transform of the form
\begin{equation}
  \bfx\shin\OO_{\text{all}} \mapsto \bfx^\eta = \bfx+\eta\bfth(\bfx), \label{shape:pert}
\end{equation}
which defines a parametrized family of domains $\OO_{\eta}(\bfth):=(\bfI\shp\eta\bfth)(\OO)$ for any given ``initial'' domain $\OO\Subset\OO_{\text{all}}$. The affine format~\eqref{shape:pert} is sufficient for defining the first-order derivatives at $\eta\sheq0$ used in this work.

\subsubsection*{Assumptions on shape transformations, admissible shapes} The set $\Ocal$ of admissible shapes for the fluid region $\OO$ in a channel period (in the wave frame) is defined as
\begin{equation}
  \Ocal = \big\{ \OO\subset\OO_{\text{all}}, \text{ $\OO$ is periodic and simply connected, }|\OO|=V_0, \ Q(\OO)=Q_0  \big\}.
  \label{Ocal:def}
\end{equation}
Accordingly, let the space $\Theta$ of admissible transformation velocities be defined as
\begin{equation}
  \Theta = \big\{ \bfth\in W^{1,\infty}(\OO_{\text{all}})\ \text{such that \ (i) }
  \bfth|_{\Gpp}=\bfth|_{\Gpm}, \ \text{(ii) }\bfth\sip\bfe_1=0 \text{ on $\Gpp$},\ \text{(iii) }\bfth(\bfz^-)=\bfze \big\},
\label{Theta:def}
\end{equation}
ensuring that the shape perturbations (i) are periodic, (ii) prevent any deformation of the end sections $\Gp^{\pm}$ along the axial direction, and (iii) prevent vertical rigid translations of the channel domain. The provision $\bfth\in W_0^{1,\infty}(\OO_{\text{all}})$ ensures that (a) there exists $\eta_0\shg0$ such that $\OO_{\eta}(\bfth)\Subset\OO_{\text{all}}$ for any $\eta\shin[0,\eta_0]$, and (b) the weak formulation for the shape derivative of the forward solution is well defined if the latter belongs to the standard solution space. Note that the volume and flow rate constraints~\eqref{vol:channel}, \eqref{flowrate} are not embedded in $\Theta$; they will be accounted for via a Lagrangian.\enlargethispage*{5ex}

\subsubsection*{Material derivatives} In what follows, all derivatives are implicitly taken at some given configuration $\OO$, i.e. at initial "time" $\eta\sheq0$. The ``initial'' material derivative $\dd{\bfa}$ of some (scalar or tensor-valued) field variable $\bfa(\bfx,\eta)$ is defined as
\begin{equation}
  \dd{\bfa}(\bfx)
 = \lim_{\eta\to0} \inv{\eta}\lsqb \bfa(\bfx^{\eta},\eta)-\bfa(\bfx,0) \rsqb \qquad\bfx\shin\OO, \label{dd:def}
\end{equation}
and the material derivative of gradients and divergences of tensor fields are given by
\begin{equation}
  \text{(a) \ }(\bfna\bfa)^{\star} = \bfna\dd{\bfa} - \bfna\bfa\sip\bfna\bfth, \qquad
  \text{(b) \ }(\Div\bfa)^{\star} = \Div\dd{\bfa} - \bfna\bfa\dip\bfna\bfth \label{dd:grad}
\end{equation}

Likewise, the first-order ``initial'' derivative $J'$ of a shape functional $J:\Ocal\to\Rbb$ is defined as
\begin{equation}
  \lbra J'(\OO),\bfth \rbra = \lim_{\eta\to0} \inv{\eta}\lpar J(\OO_{\eta}(\bfth))-J(\OO) \rpar. \label{dd:J:def}
\end{equation}

\subsubsection*{Material differentiation of integrals} Consider, for given transformation velocity field $\bfth\shin\Theta$, generic domain and contour integrals
\begin{equation}
  \text{(a) \ }I\Vsub(\eta) = \int_{\OO_{\eta}(\bfth)} F(\cdot,\eta) \dV, \qquad
  \text{(b) \ }I\Ssub(\eta) = \int_{S_{\eta}(\bfth)} F(\cdot,\eta) \ds, \label{IVS:def}
\end{equation}
where $\OO_{\eta}(\bfth)=(\bfI\shp\eta\bfth)(\OO)$ is a variable domain and $S_{\eta}(\bfth)\shdeq(\bfI\shp\eta\bfth)(S)$ a (possibly open) variable curve. The derivatives of $I\Vsub(\eta)$ and $I\Ssub(\eta)$ are given by the material differentiation identities
\begin{equation}
  \text{(a) \ } \der{I\Vsub}{\eta}\Big|_{\eta=0} = \iO \lsqb \dd{F} + F(\cdot,0)\,\Div\bfth \rsqb \dV, \qquad
  \text{(b) \ } \der{I\Ssub}{\eta}\Big|_{\eta=0} = \iS \lsqb \dd{F} + F(\cdot,0)\,\DivS\bfth \rsqb \ds, \label{dd:IVS}
\end{equation}
which are well-known material differentiation formulas of continuum kinematics. In~(\ref{dd:IVS}b), $\DivS$ stands for the tangential divergence operator, which in the present 2D context is given, with $\kappa$ denoting the curvature (see Sec.~\ref{app}) by
\begin{equation}
  \DivS\bfth = \lpar\bfI\shm\bfn\tens\bfn\rpar\dip\bfna\bfth = \partial_s\theta_s-\kappa\theta_n
\end{equation}
where we have set $\bfth$ on $\G$, using the notations and conventions introduced in~\eqref{frenet}, in the form
\begin{equation}
  \bfth=\theta_s\bftau+\theta_n\bfn. \label{theta:proj}
\end{equation}

\subsubsection*{Shape functionals and structure theorem} The structure theorem for shape derivatives (see e.g.~\cite[Chap. 8, Sec. 3.3]{del:zol}) then states that the derivative of any shape functional $J$ is a linear functional in the normal transformation velocity $\theta_n\sheq\bfn\sip\bfth|_{\dO}$. For PDE-constrained shape optimization problems involving sufficiently smooth domains and data, the derivative $\lbra J'(\OO),\bfth \rbra$ has the general form
\begin{equation}
  \lbra J'(\OO),\bfth \rbra = \idO g\, \theta_n \ds \label{shape:structure}
\end{equation}
where $g$ is the \emph{shape gradient} of $J$. The structure theorem intuitively means that (i) the shape of $\OO_{\eta}$ is determined by that of $\dO_{\eta}$, and (ii) tangential components of $\bfth$ leave $\OO_{\eta}$ unchanged at leading order $O(\eta)$.\enlargethispage*{3ex}

\subsubsection*{Example: shape derivative of channel volume constraint.} Since $|\Omega|$ is given by~(\ref{IVS:def}a) with $F=1$, this purely geometric constraint is readily differentiated using~(\ref{dd:IVS}a) and Green's theorem, to obtain
\begin{equation}
  \lbra C\Vsub'(\OO), \bfth \rbra = \iO \Div\bfth \dV = \idO \theta_n \ds = \iG \theta_n \ds, \label{dd:volume}
\end{equation}
where the last equality results from provision (ii) in~\eqref{Theta:def}.

\subsection{Shape derivative of forward solution}

The power loss and the mass flow rate constraint are expressed in terms of functionals $J\PLsub$ and $C\Qsub$ that depend on $\OO$ implicitly through the solution $(\bfu,\bff,p)$ of the forward problem. Finding their shape derivatives will then involve the solution derivative $(\bfu,\bff,p)^{\star}$. Setting up the governing problem for $(\bfu,\bff,p)^{\star}$ is then a necessary prerequisite.

\begin{prop}\label{lm1}
The governing weak formulation for the shape derivative $(\dd{\bfu},\dd{\bff},\dd{p})$ of the solution $(\bfu,\bff,p)$ of problem~\eqref{forward:weak} is
\begin{equation}
\begin{aligned}
\text{(a) \ }&&	a(\dd{\bfu},\bfv) - b(\bfv,\dd{p}) - \lbra\dd{\bff},\bfv\rbra_{\G}
 &= -a^1(\bfu,\bfv,\bfth) + b^1(\bfv,p,\bfth) + \lbra\bff,\bfv\DivS\bfth\rbra_{\G}  && \forall\bfv\shin\VS \\
\text{(b) \ }&&	b(\dd{\bfu},q) &= -b^1(\bfu,q,\bfth) && \forall q\shin\Pcal \\
\text{(c) \ }&&	\lbra\dd{\bfu}{}\Dsup,\bfg\rbra_{\G} - \lbra\dd{\bfu},\bfg\rbra_{\G} &= \lbra (\bfu\shm\bfuD)\DivS\bfth,\bfg\rbra_{\G} && \forall\bfg\shin\FS
\end{aligned} \label{der:weak}
\end{equation}
with the trilinear forms $a^1$ and $b^1$ given by
\begin{subequations}
\begin{align}
  a^1(\bfu,\bfv,\bfth)
 &= \iO 2\mu \Lcb (\bfD[\bfu]\dip\bfD[\bfv])\Div\bfth - \bfD[\bfu]\dip\lpar\bfna\bfv\sip\bfna\bfth \rpar
  - \lpar\bfna\bfu\sip\bfna\bfth \rpar\dip\bfD[\bfv] \Rcb \dV \label{a1:def} \\
  b^1(\bfu,q,\bfth)
 &= \iO q\lsqb \Div\bfu\,\Div\bfth - \bfna\bfu\dip\bfna\bfth \rsqb \dV \label{b1:def}
\end{align}
\end{subequations}
\end{prop}
\begin{proof}
The proposition is obtained by applying the material differentiation identities~\eqref{dd:IVS} to the weak formulation~\eqref{forward:weak}, assuming that the test functions are such that $\dd{\bfv}\sheq\bfze$, $\dd{\bfg}\sheq\bfze$ and $\dd{q}\sheq0$, i.e. are convected under the shape perturbation. The latter provision is made possible by the absence of boundary constraints in the adopted definition of $\VS,\FS,\Pcal$ (Sec.~\ref{Forward:weak}).
\end{proof}
Moreover, the material derivative $\dd{\bfu}{}\Dsup$ of the Dirichlet data $\bfuD$, involved in~(\ref{der:weak}c), is given in the following lemma, proved in Sec.~\ref{app}:
\begin{lemma}\label{lm4}
Let $\bfuD$ be the prescribed wall velocity~\eqref{noslip:2D}. On either component of the wall $\G$, we have
\[
  \text{(a) \ }\dd{\bfu}{}\Dsup
 = \frac{c\dd{\ell}}{L}\bftau + \frac{c\ell}{L}(\partial_s\theta_n + \kappa\theta_s)\bfn, \qquad
  \text{with \ (b) \ } \dd{\ell} = -\int_{0}^{\ell} \kappa\theta_n \ds
\]
\end{lemma}

\begin{remark}
The provision $\bfth\shin W^{1,\infty}(\OO_{\text{all}})$ in~\eqref{Theta:def} serves to ensure well-definiteness of the trilinear forms $a^1,b^1$ introduced in Proposition~\ref{lm1} for any $(\bfu,p)\shin\VS\shtimes\Pcal$.
\end{remark}
\begin{remark}
To maintain the zero-mean constraint on $p$ as $\OO$ is perturbed, the mean of $\dd{p}$ must be fixed through $\lbra \dd{p},1\rbra_{\OO}+\lbra p\Div\bfth,1\rbra_{\OO}=0$ (this provision plays no actual role in the sequel).
\end{remark}
We now state an identity involving the trilinear forms $a^1,b^1$ that will play a crucial role in the derivation of convenient shape derivative formulas for the functionals involved in this work. The proof of this result is given in Sec.~\ref{lm5:proof}\enlargethispage*{1ex}
\begin{lemma}\label{lm5}
Let $(\bfu,p)$ and $(\bfuh,\hatp)$ respectively satisfy $\Div\bfu\sheq0$, $\Div(\bfsig[\bfu,p])\sheq\bfze$ and $\Div\bfuh\sheq0$, $\Div(\bfsig[\bfuh,\hatp])\sheq\bfze$ in $\OO$, with $\bfsig[\bfu,p],\bfsig[\bfuh,\hatp]$ given for both states by the constitutive relation~(\ref{forward:PDE}b). Moreover, assume that $\bfu$, $\bfuh$ and $p$ are periodic, and set $\bff\shdeq\bfsig[\bfu,p]\sip\bfn$, $\bffh\shdeq\bfsig[\bfuh,\hatp]\sip\bfn$ and $\Delta\hatp(x_2) := \hatp(L,x_2)-\hatp(0,x_2)$ (i.e. periodicity is not assumed for $\hatp$). Then, the following identity holds:
\begin{multline}
  a^1(\bfu,\bfuh,\bfth) - b^1(\bfu,\hatp,\bfth) - b^1(\bfuh,p,\bfth) \\
 = \iG \Lcb (\bfsig[\bfu,p]\dip\bfD[\bfuh])\theta_n - \bff\sip\bfna\bfuh\sip\bfth - \bffh\sip\bfna\bfu\sip\bfth \Rcb \ds
    + \int_{\G_L} \Delta\hatp \, (\partial_2 u_1)\theta_2 \ds.
\end{multline}
\end{lemma}
\begin{proof}
See Sec.~\ref{lm5:proof}
\end{proof}

\subsection{Shape derivative of power loss functional}
\label{dd:JPL}

The derivative of the cost functional~\eqref{J:loss:def} is given, using the material differentiation identity~(\ref{dd:IVS}b), by
\begin{equation}
  \lbra J'\PLsub(\OO),\bfth \rbra
 = \iG \lsqb \bff\sip\dd{\bfu}{}\Dsup + (\bfuD\shp\bfe_1) \sip \lpar \dd{\bff} + \bff\DivS\bfth \rpar \rsqb \ds \label{dd:J:loss}
\end{equation}
in terms of the derivatives  $\dd{\bfu}{}\Dsup$ of the Dirichlet data and $\dd{\bff}$ of the stress vector. The former is given by Lemma~\ref{lm4}. The latter is part of the solution of the derivative problem~\eqref{der:weak}, whose solution therefore seems necessary for evaluating $\lbra J'\PLsub(\OO),\bfth \rbra$ in a given shape perturbation $\bfth$. However, a more effective approach allows to bypass the actual evaluation of $\dd{\bff}$ by combining the forward and derivative problems with appropriate choices of test functions:\enlargethispage*{1ex}

\begin{lemma}\label{lm2}
The following identity holds:
\[
  \iG (\bfuD\shp\bfe_1) \sip \lpar \dd{\bff} + \bff\DivS\bfth \rpar \ds
 = a^1(\bfu,\bfu,\bfth) - 2b^1(\bfu,p,\bfth) + \iG \bff\sip\dd{\bfu}{}\Dsup \ds
\]
\end{lemma}
\begin{proof}
Consider the forward problem~\eqref{forward:weak} with $(\bfv,\bfg,q)=(\dd{\bfu},\dd{\bff},\dd{p})$ and the derivative problem~\eqref{der:weak} with $(\bfv,\bfg,q)=(\bfu\shp\bfe_1,\bff,p)$. Making these substitutions and forming the combination $(\ref{der:weak}a)-(\ref{der:weak}b)-(\ref{der:weak}c)-(\ref{forward:weak}a)+(\ref{forward:weak}b)+(\ref{forward:weak}c)$, one obtains
\[
  \iG (\bfuD\shp\bfe_1)\sip\dd{\bff} \ds - \iG \dd{\bfu}{}\Dsup\sip\bff \ds
 = a^1(\bfu,\bfu,\bfth) - 2b^1(\bfu,p,\bfth)
 - \iG \bff\sip (2\bfu\shp\bfe_1\shm\bfuD)\DivS\bfth \ds
\]
The Lemma follows from using $\bfu\sheq\bfuD$ on $\G$ and rearranging the above equality.
\end{proof}

Then, using Lemma~\ref{lm2} in~\eqref{dd:J:loss} and evaluating $a^1(\bfu,\bfu,\bfth) - 2b^1(\bfu,p,\bfth)$ by means of Lemma~\ref{lm5} with $(\bfuh,\bffh,\hatp)=(\bfu\shp\bfe_1,\bff,p)$ (implying that $\Delta p\sheq0$), the derivative of $J\PLsub$ is recast in the following form, which no longer involves the solution of the derivative problem:
\begin{equation}
  \lbra J'\PLsub(\OO),\bfth \rbra
 = \iG \Lcb (2\mu\bfD[\bfu]\dip\bfD[\bfu])\theta_n + 2\bff\sip(\dd{\bfu}{}\Dsup-\bfna\bfu\sip\bfth) \Rcb \ds
\label{dd:J:loss:trial}
\end{equation}
This expression is still somewhat inconvenient as it involves (through $\bfD[\bfu]$) the complete velocity gradient on $\G$. This can be alleviated by using the decomposition
\begin{equation}
  \bfna\bfu = \nS\bfu + \partial_n\bfu\tens\bfn \label{DS+Dn}
\end{equation}
of the velocity gradient (where $\nS\bfu$ and $\partial_n\bfu$ respectively denote the tangential gradient and the normal derivative of $\bfu$) and expressing $\partial_n\bfu$ in terms of $\bff$ by means of the constitutive relation~(\ref{forward:PDE}b). In view of the specific form~\eqref{noslip:2D} of the Dirichlet data, the latter step is here conveniently carried out explicitly (in Sec.~\ref{proof:lemmas}), using curvilinear coordinates, and yields:
\begin{lemma}\label{forward:wall}
Let $(\bfu,\bff,p)$ solve the forward problem~\eqref{forward:weak}. On the channel wall $\G$, we have
\begin{equation}
  \text{(a) \ }\bfna\bfu
 = \frac{\kappa c\ell}{L} \bfn\tens\bftau
 + \Lpar \frac{f_s}{\mu}-\frac{\kappa c\ell}{L} \Rpar \bftau\tens\bfn, \quad
  \text{(b) \ }2\bfD[\bfu] = \frac{f_s}{\mu} \lpar \bfn\tens\bftau \shp \bftau\tens\bfn \rpar, \quad
  \text{(c) \ }\bff = -p\bfn + f_s\bftau
\end{equation}
where $f_s\shdeq\bff\sip\bftau$ (hence the viscous part of $\bff$ is tangential to $\G$).
\end{lemma}
Finally, we evaluate the density of the integral in~\eqref{dd:J:loss:trial} using the formulas of Lemmas~\ref{lm4} and~\ref{forward:wall}. On performing traightforward algebra and rearranging terms, we obtain the following final result for $\lbra J'\PLsub(\OO),\bfth \rbra$, which is suitable for a direct implementation using the output of a boundary integral solver:

\begin{prop}\label{dJ:final}
The shape derivative of the power loss cost functional $J\PLsub$ in a shape perturbation whose transformation velocity field $\bfth$ satisfies asumptions~\eqref{Theta:def} is given (with $f_s\shdeq\bff\sip\bftau$) by
\begin{align}
  \lbra J'\PLsub(\OO),\bfth \rbra
 &= \sum_{\eps=+,-} \int_{\G^{\eps}} \Lcb \Lsqb 2\frac{c\ell^{\eps}\kappa}{L}f_s -  \inv{\mu} f_s^2
   \Rsqb\theta_n
 + \frac{2c}{L} \lpar \dd{\ell}{}^{\eps} f_s - \ell^{\eps}(\partial_s\theta_n)p \rpar \Rcb \ds.
\label{dd:J:loss:result2}
\end{align}
where $\eps\shin\{+,-\}$ refers to the upper wall $\G^+$ or the lower wall $\G^-$.
\end{prop}

\begin{remark}
The shape derivative $\lbra J'\PLsub(\OO),\bfth \rbra$ as given by Proposition~\ref{dJ:final} is a linear functional on $\theta_n|_\G$ (since $\dd{\ell}$ is one, see Lemma~\ref{lm4}), as predicted by the structure theorem. Moreover, it is insensitive to the pressure being possibly defined up to an additive constant: indeed, replacing $p$ with $p\shp\Delta p$ adds
\begin{equation}
  - \frac{2c\ell^+}{L} \int_{\G^{+}} \partial_s\theta_n \ds - \frac{2c\ell^-}{L} \int_{\G^{-}} \partial_s\theta_n \ds
\end{equation}
to $\lbra J'\PLsub(\OO),\bfth \rbra$ as given by Proposition~\ref{dJ:final}, and this quantity vanishes due to the assumed periodicity of $\theta_n$.
\end{remark}

\subsection{Shape derivative of mass flow rate functional}
\label{dd:CQ}

The derivative of the mass flow rate functional $C\Qsub$, defined by~\eqref{flowrate}, is given (recalling~\eqref{dd:volume} for the shape derivative of $|\OO|$) by
\begin{equation}
  \lbra C\Qsub'(\OO), \bfth \rbra
 = \lbra I'(\OO), \bfth \rbra + \frac{c}{L}\iG \theta_n \ds, \qquad 
  \text{with} \ I(\OO) := \inv{L}\iO \bfu_{\OO}\sip\bfe_1 \dV. \label{dd:flowrate}
\end{equation}
To evaluate $\lbra I'(\OO), \bfth \rbra$, with $\bfu$ solving the forward problem~\eqref{forward:PDE}, we recast $I(\OO)$ as an integral over the end section $\G_0$ by means of the divergence theorem, recalling that $\bfu_{\OO}$ is divergence-free, tangential on $\G$ and periodic:
\[
  I(\OO)
 = \inv{L}\iO \lsqb \Div(x_1\bfu_{\OO}) - x_1\Div\bfu_{\OO} \rsqb \dV
 = \inv{L}\Lcb \int_{\G_L} x_1(\bfu_{\OO}\sip\bfe_1) \ds - \int_{\G_0} x_1(\bfu_{\OO}\sip\bfe_1) \ds \Rcb
 = \int_{\G_L} u_{\OO}\sip\bfe_1 \ds.
\]
Consequently, using identity~(\ref{dd:IVS}b) with provision (ii) of~\eqref{Theta:def}, we find:
\begin{equation}
  \lbra I'(\OO), \bfth \rbra
 = \int_{\G_L} \lpar \dd{\bfu} \shp \bfu \DivS\bfth \rpar \sip \bfe_1 \ds
 = \int_{\G_L} \lpar \dd{u}_1 \shp u_1 \partial_2\theta_2 \rpar \dx_2 \label{dd:I0}
\end{equation}
The forward solution derivative $\dd{\bfu}$ in the above expression will now be eliminated with the help of an adjoint problem. Let the adjoint state $(\bfuh,\bffh,\hatp)$ be defined as the solution of the mixed weak formulation
\begin{equation}
\begin{aligned}
\text{(a) \ }&&
  a(\bfuh,\bfv) - b(\bfv,\hatp) - \lbra \bffh ,\bfv \rbra_{\G} &= -\lbra 1 , \bfv\sip\bfe_1 \rbra_{\Gpp}
 && \forall\bfv\shin\VS, \\
\text{(b) \ }&&  b(\bfuh,q) &= 0 && \forall q\shin\Pcal, \\
\text{(c) \ }&& - \lbra\bfuh,\bfg\rbra_{\G} &= 0 && \forall\bfg\shin\FS.
\end{aligned}\label{adjoint:weak}
\end{equation}
The adjoint state $(\bfuh,\bffh,\hatp)$ thus results from applying a unit pressure difference $\Delta\hatp=1$ between the channel end sections while prescribing a no-slip condition on the channel walls, i.e. Problem~(2.14--2.18) of~\cite{barnett:16} with homogeneous Dirichlet data on the walls (see Remark~\ref{pressure:diff}).\enlargethispage*{1ex}

Then, combining the adjoint problem~\eqref{adjoint:weak} with $(\bfv,\bfg,q)=(\dd{\bfu},\dd{\bff},\dd{p})$ and the derivative problem~\eqref{der:weak} with $(\bfv,\bfg,q)=(\bfuh,\bffh,\hatp)$ and recalling the no-slip condition $\bfuh=\bfze$ on $\G$ yields the identity
\begin{equation}
  \int_{\G_L} \dd{u}_1 \dx_2
   = a^1(\bfu,\bfuh,\bfth) - b^1(\bfuh,p,\bfth) - b^1(\bfu,\hatp,\bfth)
   + \iG \dd{\bfu}{}\Dsup\sip\bffh \ds. \label{dd:I}
\end{equation}
We next use this identity in~\eqref{dd:I0} and substitute the resulting expression of $\lbra I'(\OO), \bfth \rbra$ into~\eqref{dd:flowrate}, to obtain
\begin{equation}
  \lbra C\Qsub'(\OO), \bfth \rbra
 = a^1(\bfu,\bfuh,\bfth) - b^1(\bfuh,p,\bfth) - b^1(\bfu,\hatp,\bfth)
 + \iG \Lpar \frac{c}{L}\theta_n + \bffh\sip\dd{\bfu}{}\Dsup \Rpar \ds
 + \int_{\G_L} u_1 \partial_2\theta_2 \dx_2 \label{aux5}
\end{equation}
We first apply Lemma~\ref{lm5} with $\Delta\hatp=1$ to the above expression, to obtain
\begin{equation}
  \lbra C\Qsub'(\OO), \bfth \rbra
 =  \iG \Lpar -(2\mu\bfD[\bfu]\dip\bfD[\bfuh])\theta_n + \frac{c}{L}\theta_n
 - \bff\sip\bfna\bfuh\sip\bfth
 - \bffh\sip(\bfna\bfu\sip\bfth-\dd{\bfu}{}\Dsup) \Rpar \ds + \int_{\G_L} \partial_2(u_1\theta_2) \dx_2, \label{dd:I:exp}
\end{equation}
whose form is similar to~\eqref{dd:J:loss:trial}. The first integral in~\eqref{dd:I:exp} can be reformulated with the help of Lemma~\ref{lm4}, Lemma~\ref{forward:wall} and its following counterpart for the adjoint solution (whose proof is given in Sec.~\ref{proof:lemmas}):
\begin{lemma}\label{adjoint:wall}
Let $(\bfuh,\bffh,\hatp)$ solve the adjoint problem~\eqref{adjoint:weak}. On the channel wall $\G$, we have
\begin{equation}
  \text{(a) \ }\bfna\bfuh = \frac{\fhat_s}{\mu}\bftau\tens\bfn, \qquad
  \text{(b) \ }2\bfD[\bfuh] = \frac{\fhat_s}{\mu} \lpar \bfn\tens\bftau \shp \bftau\tens\bfn \rpar, \qquad
  \text{(c) \ }\bffh = -\hatp\bfn + \fhat_s\bftau
\end{equation}
where $\fhat_s\shdeq\bffh\sip\bftau$ (hence the viscous part of $\bffh$ is tangential to $\G$).
\end{lemma}
Carrying out these derivations, and evaluating the second integral of~\eqref{dd:I:exp}, establishes the following final, implementation-ready, formula for $\lbra C\Qsub'(\OO), \bfth \rbra$:
\begin{prop}\label{dd:massflow}
The shape derivative of the mass flow rate functional $C\Qsub$ in a domain shape perturbation such that the transformation velocity field $\bfth$ satisfies asumptions~\eqref{Theta:def} is given by
\begin{align}
  \lbra C'\Qsub(\OO),\bfth \rbra
 &=  \sum_{\eps=+,-} \int_{\G^{\eps}} \Lcb \Lpar \frac{c\ell^{\eps}\kappa}{L} -  \inv{\mu} f_s\fhat_s + \frac{c}{L} \Rpar \theta_n
  + \frac{c}{L} \lpar \dd{\ell}{}^{\eps} \fhat_s -  \ell^{\eps}(\partial_s\theta_n)\phat \rpar \Rcb \ds \suite\qquad
  + [\theta_2 u_1](\bfz^+) - [\theta_2 u_1](\bfz^-)
\end{align}
where $f_s$ and $(\hatp,\fhat_s)$ are components of the solutions of the forward problem~\eqref{forward:weak} and the adjoint problem~\eqref{adjoint:weak}, respectively, and $\eps\shin\{+,-\}$ refers to the upper wall $\G^+$ or the lower wall $\G^-$.
\end{prop}

\begin{remark}\label{pressure:diff}
Let the Stokes solution $(\bfuh,\hatp)$, periodic up to a constant (unit) pressure drop, be defined by
\begin{equation}
  -\Div\lpar 2\mu\bfD[\bfuh]\shm\hatp\bfI\rpar = \bfze \ \ \text{in $\OO$}, \quad \Div\bfuh\sheq0, \quad\bfu\sheq\bfze \ \ \text{on $\G$}, \quad
  \hatp(\cdot)\shm \hatp(\cdot\shm L\bfe_1) = 1 \ \ \text{on $\Gp^+$}
\end{equation}
(i.e. Problem~(2.14--2.18) of~\cite{barnett:16} with homogeneous Dirichlet data on the walls). Taking the dot product of the first equation by $\bfv\shin\VS$, integrating by parts and using the periodicity of $\bfv$ and $2\mu\bfD[\bfu]$, we obtain equation~(\ref{adjoint:weak}a), while equations~(\ref{adjoint:weak}b,c) express the incompressibility and no-slip conditions.
\end{remark}

\section{Numerical algorithm} \label{sc:numerical}
We now formally define the shape optimization problem and describe a numerical algorithm to solve it based on the shape sensitivity formulas derived in the previous section.  

\subsection{Optimization method}

Our goal is to find the shape of the peristaltic pump in the wave frame that minimizes the power loss functional subject to constant volume and flow rate constraints, that is, 
\begin{equation}
 \OO^{\star} =  \argmin_{\OO \, \shin\, \Ocal} \; J\PLsub(\OO) \quad\text{subject to \ } C\Qsub(\OO) = 0, C\Vsub(\OO) = 0,
\label{eq:opt}\end{equation}
with $\Ocal$ defined as in \eqref{Ocal:def}. While there are numerous approaches for solving constrained optimization problems \cite{NoceWrig06}, we use an Augmented Lagrangian (AL) approach and avoid second-order derivatives of the cost functional, whose evaluation is somewhat challenging in the case of our shape optimization problem.  It proceeds by forming an augmented Lagrangian, defined by
\begin{equation}
  \Lcal\Asub(\OO,\lambda;\sigma)
 = J\PLsub(\OO) - \lambda_1 C\Qsub(\OO) - \lambda_2 C\Vsub(\OO) + \frac{\sigma_1}{2}C^2\Qsub(\OO) + \frac{\sigma_2}{2}C^2\Vsub(\OO),
\label{eq:augla}\end{equation}
where $\sigma = (\sigma_1, \sigma_2)$ are penalty coefficients that are positive and  $\lambda = (\lambda_1, \lambda_2)$ are Lagrange multipliers. Setting the initial values $\sigma^0$ and $\lambda^0$ using heuristics, the augmented Lagrangian method introduces a sequence $(m= 1, 2,\dots)$ of unconstrained minimization problems:
\begin{equation}
 \OO_m =  \argmin_{\OO \, \shin\, \Ocal}  \Lcal\Asub(\OO,\lambda^m;\sigma^m) ,
\label{eq:minsubp}\end{equation}
with explicit Lagrange multiplier estimates $\lambda^m$ and increasing penalties $\sigma^m$. We use the Broyden-Fletcher-Goldfarb-Shanno (BFGS) algorithm \cite{NoceWrig06}, a quasi-Newton method, for solving ~\eqref{eq:minsubp}. Equation~\eqref{dd:volume}, Propositions \ref{dJ:final} and \ref{dd:massflow} are used in this context for all gradient evaluations. 

The overall optimization procedure is summarized in Algorithm~\ref{alg:AGM}.

\begin{algorithm}[tb] \label{alg:AGM}
\caption{\small\em Augmented Lagrangian method for problem~\eqref{eq:opt}}
\algsetup{indent=3em}
\begin{algorithmic}[1]
\STATE Choose initial design shape parameters $\bfxi^0 $, which determines initial domain $\OO_0$ \\ Set convergence tolerance $\zeta^{\star} = 10^{-3}$ \\ Set $\lambda^0 = (0,0)$ \\ Set $\sigma^0_1=10$, $\sigma^0_2 = 10$ or $\sigma^0_2 = 100$ (depending on $\OO_0$) 
\\ Set $\zeta^1= (\sigma^0)^{-0.1}$
\FOR{$m = 1, 2, \dots$}
  \STATE \# Solve unconstrained minimization subproblem~\eqref{eq:minsubp} for $\OO_m$\vspace*{1ex}
  \STATE From $\bfxi^{m-1} = \bfxi^{m-1,0}$ and an approximate (positive definite) Hessian matrix $B_0$, repeat the following steps until $\bfxi^{m-1,j}$ converges to the solution $\bfxi^m$:
  \STATE 1) Obtain a perturbation direction $\bfp_{j}$ by solving $B_j \bfp_{j} = - \nabla_{\bfxi} \Lcal\Asub(\bfxi^{m-1,j} , \lambda^{m-1};\sigma^{m-1})$
  \STATE 2) find an acceptable stepsize $\eta_{j}$
  \STATE 3) $\bfxi^{m-1,j+1} = \bfxi^{m-1, j} + \eta_j \bfp_{j}$
  \STATE 4) Update the approximate Hessian matrix $B_{j+1} $(by BFGS formula)
   \STATE  The solution $\bfxi^m$ uniquely determines $\OO_m$.\vspace*{1ex}
   \STATE \# Check for convergence and/or update Lagrange multipliers and penalty parameters\vspace*{1ex}
  \IF{$|C\Vsub(\OO_m)|\shl\zeta^{m}_1$ and $|C\Qsub(\OO_m)|\shl\zeta^{m}_2$ }
    \IF{$|C\Vsub(\OO_m)|\shl\zeta^{\star}$ and $|C\Qsub(\OO_m)|\shl\zeta^{\star}$}
      \STATE Set $\bfxi^{\star}:=\bfxi^m$, i.e., $\OO^{\star}:=\OO_m$ \quad \textbf{STOP}
    \ENDIF
    \STATE{$\lambda^{m}_1=\lambda^{m-1}_1-\sigma^{m-1}_1 C\Vsub(\OO_m)$}
    \STATE{$\lambda^{m}_2=\lambda^{m-1}_2-\sigma^{m-1}_2 C\Qsub(\OO_m)$}
    \STATE{$\sigma^{m}=\sigma^{m-1}$}
    \STATE{$\zeta^{m+1}=(\sigma^0)^{-0.9}\zeta^m$}
  \ELSE
    \STATE{$\lambda^{m}=\lambda^{m-1}$}
    \STATE{$\sigma^{m}=10\sigma^{m-1}$}
    \STATE{$\zeta^{m+1}=(\sigma^0)^{-0.1}$}
  \ENDIF
\ENDFOR
\end{algorithmic}
\end{algorithm}

\subsection{Finite-dimensional parametrization of shapes} In view of both the structure theorem, see~\eqref{shape:structure}, and the fact that we rely for the present study on a boundary integral method, we only need to model shape perturbations of the channel walls. Here we consider, for each wall $\G^{\pm}$, parametrizations of the form $\bfx^{\pm} = \bfx^{\pm}(t, \bfxi)$ with 
%
\begin{align}
 \G\ni x_1^{\pm} (t) &:= \frac{L}{2\pi} t -\sum_{k=1}^{N} \xi_{1,k}^{\pm} + \sum_{k=1}^{2N} \xi^{\pm}_{1,k}  \phi_k (t) \quad t\shin [0, 2\pi],
\label{eq:Param2a}\\
 \G\ni x_2^{\pm} (t) &:=  \xi_{2,0}^{\pm}-\sum_{k=1}^{N} \xi_{2,k}^{\pm} + \sum_{k=1}^{2N} \xi^{\pm}_{2,k}  \phi_k (t) \quad t\shin [0, 2\pi],
\label{eq:Param2b}
\end{align}
where $\phi_k (t)$ are the trigonometric polynomials $\{\cos(t), \cos(2t), \dots, \cos(Nt), \sin(t), \sin(2t), \dots, \sin(Nt)\}$. 
The set of admissible shapes as in \eqref{Ocal:def} indicates $x_1^{\pm}(0) = 0$ and $x_1^{\pm}(2\pi) = L$, which are enforced in \eqref{eq:Param2a}.
The constraint (iii) of~\eqref{Theta:def} is fulfilled by a fixed value $\xi_{2,0}^{-} = x_2^{-}(0)$ in  \eqref{eq:Param2b}, which will be pre-assigned and excluded from the shape parameters.

Therefore, we take as design shape parameters the set $\bfxi = \{\xi^{\pm}_{1,1},\cdots,\xi^{\pm}_{1,2N}, \xi^{+}_{2,0}, \xi^{\pm}_{2,1}, \cdots, \xi^{\pm}_{2, 2N}\} $ of dimension $8N+1$. Since the parametrization~\eqref{eq:Param2a}, \eqref{eq:Param2b} is linear in $\bfxi$, transformation velocities $\bfth$ associated to perturbed parameters $\bfxi(\eta) \sheq \bfxi \shp \eta \bfp$, i.e., to the $\eta$-dependent parametrization have the form
\begin{equation}
    \bfth(\bfx(t)) = \frac{1}{\eta} \left( \bfx(t; \bfxi + \eta \bfp ) -   \bfx(t; \bfxi ) \right) \quad t\shin [0, 2\pi].
\end{equation}

\subsection{Boundary integral solver} \label{sec:forwardsolver}

Solving the unconstrained optimization problem~\eqref{eq:minsubp} requires evaluating the shape sensitivites (Propositions \ref{dJ:final} and \ref{dd:massflow}), which in turn require solving the forward and adjoint problems to obtain the corresponding traction and pressure on the channel walls. In both problems, the fluid velocity and pressure satisfy the Stokes equations:
\begin{equation}
-\nabla p + \mu \Delta \bfu = 0 \quad \text{and} \quad \nabla \cdot \bfu = 0 \quad\text{in} \quad \Omega.
\label{eq:Stokes2} \end{equation}
While the forward problem requires applying prescribed slip on the walls and periodic boundary conditions, the adjoint problem requires a no-slip on the walls and unit pressure drop across the channel. The boundary conditions can be summarized in a slightly modified form as follows.
%
\begin{eqnarray}
\mbox{\textit{Forward problem:} } &\quad&  \bfu = \bfuD \, \text{on $\G$},  \quad \bfu |_{\Gamma_L} - \bfu |_{\Gamma_0} = \mathbf{0} \quad \text{and} \quad T |_{\Gamma_L} - T |_{\Gamma_0} = \mathbf{0}.  \label{eq:forwprb} \\
\mbox{\textit{Adjoint problem:} } &\quad& \bfu = \mathbf{0} \,\, \text{on $\G$},  \quad \bfu |_{\Gamma_L} - \bfu |_{\Gamma_0} = \mathbf{0} \quad \text{and} \quad T |_{\Gamma_L} - T |_{\Gamma_0} = \bfe_1.   \label{eq:adjprb}
\end{eqnarray}
Here, $T$ is the {\em traction} vector, whose components are given by $T_i(\bfu, p) = \sigma_{ij}(\bfu, p) \bfn_j$ and it represents the hydrodynamic force experienced by any interface in the fluid with normal $\bfn$. These systems of equations correspond to \eqref{forward:PDE} and \eqref{adjoint:weak} respectively by unique continuation of Cauchy data \cite{barnett:16}.  The standard boundary integral approach for periodic flows is to use periodic Greens functions obtained by summing over all the periodic copies of the sources \cite{poz}. The disadvantage of this approach is the slow convergence rate, specially, for high-aspect ratio domains; moreover, the pressure-drop condition cannot be applied directly. Instead, we use the periodization scheme developed recently in \cite{barnett:16} that uses the free-space kernels only and enforces the inlet and outlet flow conditions in (\ref{eq:forwprb}, \ref{eq:adjprb}) algebraically at a set of collocation nodes. 

The free-space Stokes single-layer kernel and the associated pressure kernel, given a source point $\y$ and a target point $\x$, are given by
\begin{equation}
S(\x,{\bf y})=\frac{1}{4\pi\mu}\left(-\log{|\x - \y|} \, \mathbf{I} + \frac{(\x - \y)\otimes(\x - \y)}{|\x-\y|^2}\right) \quad\text{and}\quad Q({\bf x},{\bf y})=\frac{1}{2\pi} \frac{\x-\y}{|\x-\y|^2}.
\label{S}
\end{equation}
The approach of \cite{barnett:16} represents the velocity field as a sum of free-space potentials defined on the unit cell $\Omega$, its nearest periodic copies and at a small number, $K$, of auxiliary sources located exterior to $\Omega$ that act as proxies for the infinite number of far-field periodic copies: 
\begin{equation}
\bfu \;=\; {\cal S}_\Gamma^\text{near}\bm\tau + \sum_{m=1}^K {\bf c}_m \phi_m,
\label{urep}
\end{equation}
where
\begin{equation}
({\cal S}_\Gamma^\text{near}\bm\tau)(\x)
:=
\sum_{|n|\leq 1}{\int_\Gamma{\! S({\bf x},{\bf y}+n{\bf d})\bm{\tau}({\bf y})\, ds_{\bf y}}} \quad\text{and}\quad \phi_m (\bfx) = S(\x, \y_m). 
\label{Snr}
\end{equation}
Here, ${\bf d}$ is the lattice vector i.e.\ ${\bf d} = L{\bf e}_1$ and  the source locations $\{\y_m\}_{m=1}^K$ are chosen to be equispaced on a circle enclosing $\Omega$. The associated representation for pressure is given by
\begin{equation}
p \;=\; {\cal P}_\Gamma^\text{near}\bm\tau + \sum_{m=1}^K {\bf c}_m \cdot {\bf \varphi}_m, 
\label{prep}
\end{equation}
where 
\begin{equation}
({\cal P}_\Gamma^\text{near}\bm\tau)({\bf x})
:=
\sum_{|n|\leq 1}{\int_\Gamma{\! Q({\bf x},{\bf y}+n{\bf d})\cdot\bm{\tau}({\bf y})\, ds_{\bf y}}} \quad\text{and}\quad {\bf \varphi}_m (\x) = Q(\x, \y_m). 
\label{Pnr}
\end{equation}
The pair $(\bfu, p)$ as defined by this representation satisfy the Stokes equations since $S$ and $Q$, defined in \eqref{S}, are the Green's functions. The unknown density function $\bm\tau$ and the coefficients $\{{\bf c}_m\}$ are then determined by enforcing the boundary conditions. For the forward problem, applying the conditions in \eqref{eq:forwprb} produces a system of equations in the following form \cite{barnett:16}: 
\begin{equation}
\left[ \begin{array}{cc}
A & B\\
C & D\\
\end{array}\right]
\left[ \begin{array}{c}
{\bm \tau}\\
{\bm c}\\
\end{array}\right]=
\left[ \begin{array}{c}
\bfuD \\
{\bm 0}\\
\end{array} \right].
\label{bsys} \end{equation}
The first row applies the slip condition on $\Gamma$ by taking the limiting value of $\bfu(\bfx)$, defined in \eqref{urep}, as $\x$ approaches $\Gamma$ from the interior. The second row applies the periodic boundary conditions on velocity and traction as defined in \eqref{eq:forwprb}. The operators $A, B, C$ and $D$ are correspondingly defined based on the representation formulas \eqref{urep} and \eqref{prep}. In the case of the adjoint problem, the operators remain the same but the right hand side of \eqref{bsys} is modified according to  the boundary conditions \eqref{eq:adjprb}.  

\begin{remark}\label{SLP:remark}
Note that the representation \eqref{urep} implies that $A$ is a first-kind boundary integral operator and in general is not advisable for large-scale problems due to ill-conditioning of resulting discrete linear systems. However, for the problems we consider here, the dimension of linear system is usually small $\sim$100-200 and we always solve it using direct solvers. For large-scale problems requiring iterative solution (e.g., in peristaltic pumps transporting rigid or deformable particles), well-conditioned systems can be produced from second-kind operators, which can readily be constructed using double-layer potentials as was done in \cite{barnett:16}.
\end{remark}

Finally, we use $M$ quadrature nodes each on $\Gamma^+$ and $\Gamma^-$ to evaluate smooth integrals  using the standard periodic trapezoidal rule and weakly singular integrals, such as \eqref{Snr} evaluated on $\Gamma$, using a spectrally-accurate Nystr\"{o}m method (with periodic Kress corrections for the log singularity, see Sec. 12.3 of \cite{kress1999linear}). Thereby, a discrete linear system equivalent to \eqref{bsys} is obtained, which is solved using a direct solver for the unknowns $\bm\tau$ and $\{{\bf c}_m\}$. 
The traction vector for a given pipe shape, required for evaluating the shape derivatives (Props. \ref{dJ:final}  and \ref{dd:massflow}), can then be obtained by a similar representation as \eqref{urep} but with the kernel replaced by the traction kernel, given by, 
\begin{equation}
T(S, Q)(\x,{\bf y}) \;=\; -\frac{1}{\pi} \frac{(\x - \y)\otimes(\x - \y)}{|\x-\y|^2} \frac{(\x - \y)\cdot{\bf n}^\x}{|\x-\y|^2}.
\label{Tker}
\end{equation}

\section{Results} \label{sc:results}
In this section, we present validation results for our numerical PDE solver and test the performance of our shape optimization algorithm. 
\begin{figure}[htbp] \centering
\includegraphics[width=\textwidth]{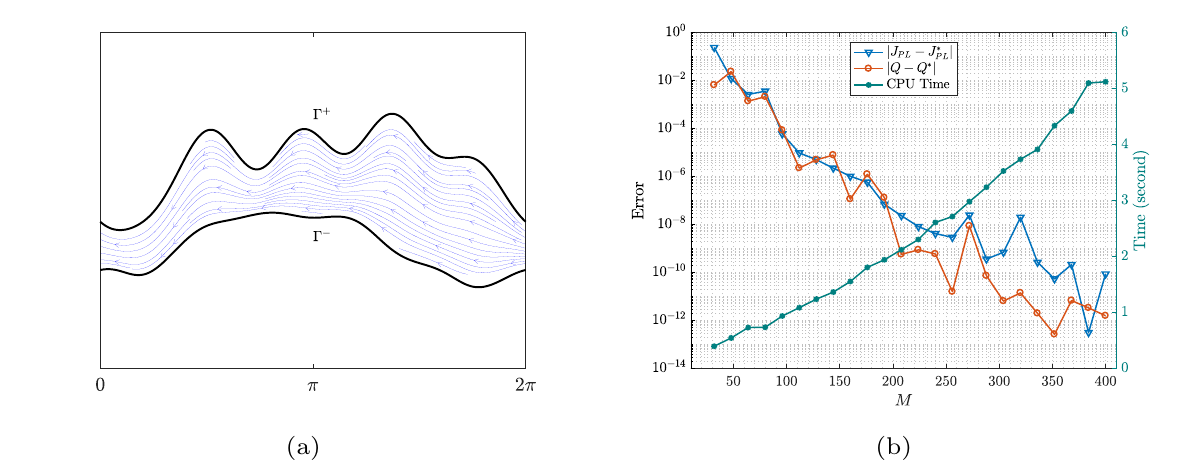}
\caption{(a) An arbitrarily shaped channel and the streamlines of flow induced by a prescribed slip on the walls, as defined in \eqref{forward:PDE}, obtained using our boundary integral solver. (b) Plot of convergence and CPU times as a function of the spatial resolution $M$ used in the forward problem solver. The reference solution (denoted by superscript $^*$) is obtained using $M=1024$. Notice that we get single-precision accuracy ($1e-6$) even with a small number of points ($\sim$100) and corresponding cost of forward solve is less than a second. }
\label{fig:forward}
\end{figure}
First, we demonstrate the performance of our forward problem \eqref{forward:PDE} solver on an arbitrary pump shape as shown in Figure \ref{fig:forward}(a). To illustrate the convergence of the numerical scheme, we consider two scalar quatities: the objective function  $J_{PL}$ and the mass flow rate $Q$, both of which depend on the traction vector obtained by solving the forward problem. In \ref{fig:forward}(b), we show the accuracy of the solver in computing these quantities as well as the CPU time it takes to solve as the number of quadrature points on each of the walls, $M$, is increased. Since we are using a spectrally-accurate quadrature rule, notice that the error decays rapidly and a small number of points are sufficient to achieve six-digit accuracy, which is more than enough for our application. The number of proxy points and the number of collocation points on the side walls are chosen following the analysis of \cite{barnett:16}---both typically are small again owing to the spectral convergence of the method w.r.t these parameters. Consequently, each forward solve takes less than a second to obtain the solution to six-digit accuracy as can be observed from Figure \ref{fig:forward}(b). 

\begin{figure}[htbp] \centering
\includegraphics[width=0.8\textwidth]{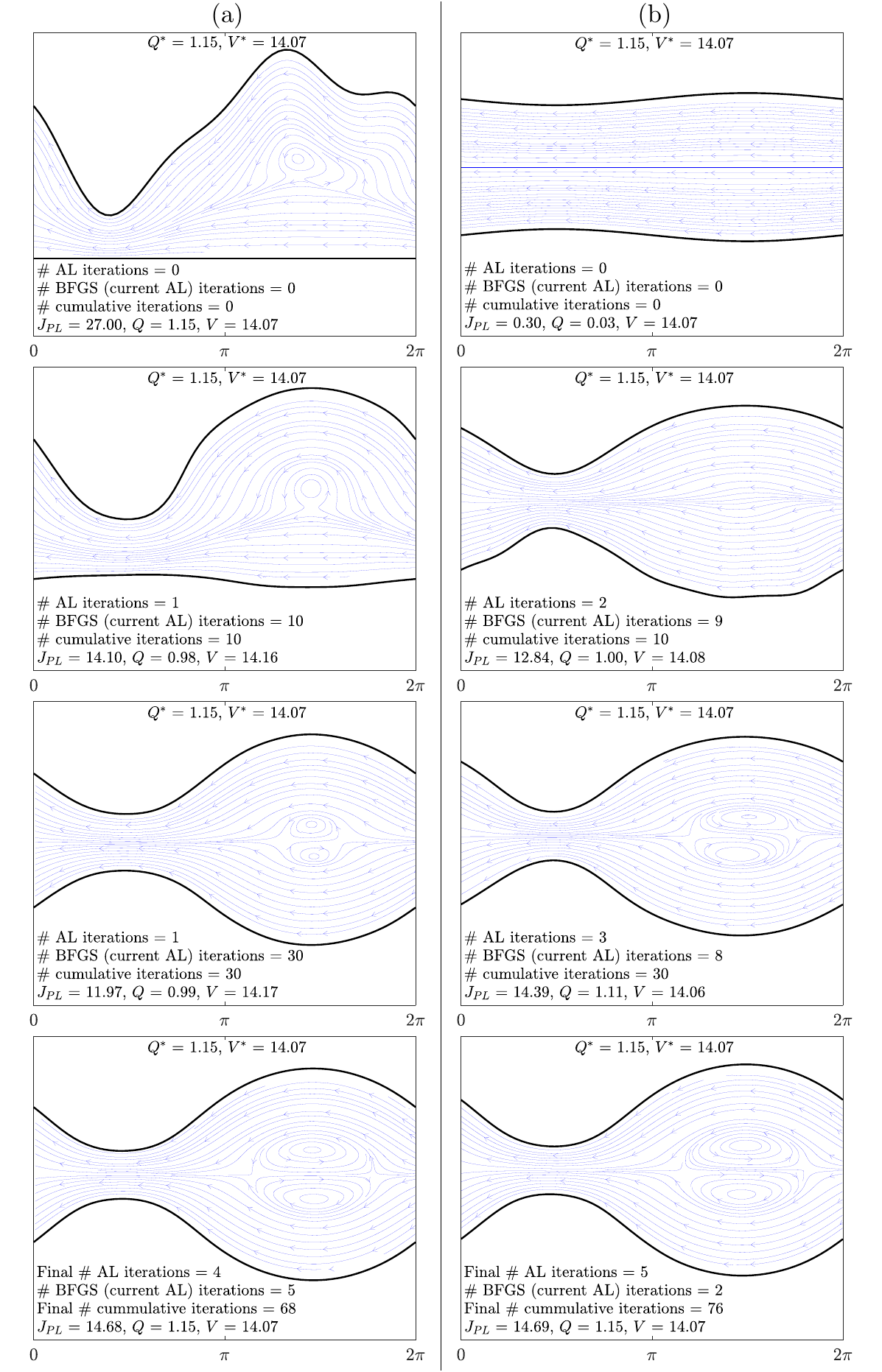}
\caption{Optimizing power loss with mass flow rate and volume constraints, starting with two different initial shapes, (a) a random top wall with a flat bottom wall, and (b) a symmetric slight bump shape. In both cases, the constraints on $Q^*$ and $V^*$ are the same. We observe that both arrive at the same equilibrium shape. }
\label{fig:2shape_iteration}
\end{figure}

Next we consider two different initial pump shapes and perform the shape optimization by imposing the same constraints on the flow rate $Q^*$ and the volume $V^*$. The initial and the intermediary shapes produced by our numerical optimization procedure are shown in Figure \ref{fig:2shape_iteration}.  Based on our earlier analysis of the forward solver, we set $M=64$ and the number of modes $N$ (e.g., in \eqref{eq:Param2a} and \eqref{eq:Param2b}) is set to 5, thereby, the number of the shape design parameters is 41. We use the Augmented Lagrangian approach described in Algorithm \ref{alg:AGM} with the values of the penalty parameters as listed. Each AL iteration requires solving an unconstrained optimization problem, which entails taking several BFGS iterations (steps 5-9 of Algo \ref{alg:AGM}). Both the number of AL iterations and the number of BFGS iterations are shown for each shape update in Figure \ref{fig:2shape_iteration}. In total, it costs 143 solves of the forward and adjoint problems \eqref{eq:forwprb}--\eqref{eq:adjprb} for the case in Fig.~\ref{fig:2shape_iteration}(a) and 197 for the case in Fig.~\ref{fig:2shape_iteration}(b).

We show the evolution of the objective function and the constraints with the iteration index, corresponding to these two test cases, in Fig.~\ref{fig:2shape_JQV}. Due to the fact that we are using an AL approach, the constraints are enforced progressively by increasing the penalty coefficients $\sigma$ and as a result, the objective function approaches a local minimum in a non-monotonic fashion. In the second test case (Fig. \ref{fig:2shape_iteration}b), for instance, $J_{PL}$ increases significantly since the initial shape doesn't satisfy the constraints. On the other hand, in the first test case, the initial values for the $Q$ and $V$, obtained by a forward solve, match the target values $Q^*$ and $V^*$, thereby $J_{PL}$ is reduced as the optimization proceeds, to nearly half of its initial value. While not guaranteed in general (due to the possibility of getting stuck in a local minimum), the final shapes in both cases coincide, which is another validation of our analytic shape sensitivity calculations.   

\begin{figure}[htbp] \centering
\includegraphics[width=\textwidth]{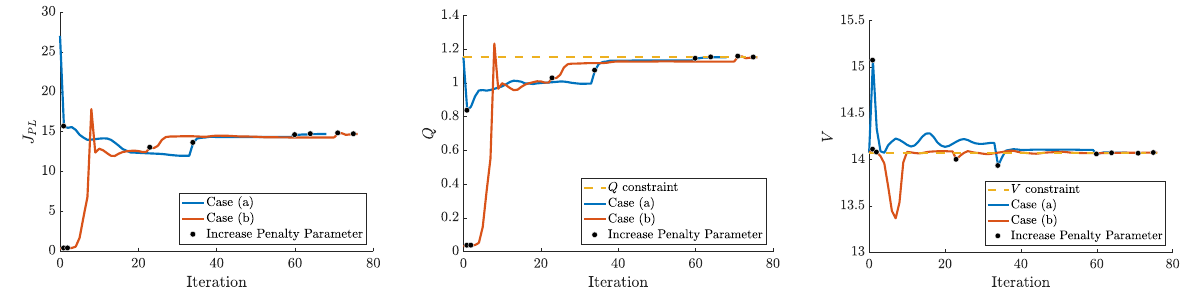}
\caption{Evolution of the power loss, mass flow rate, and volume as a function of the optimization iteration index corresponding to the two test cases shown in Fig.~\ref{fig:2shape_iteration}. Since we are using an Augmented Lagrangian approach, as expected, we can observe that the constraints---of prescribed mass flow rate and volume indicated by dashed lines---are satisfied progressively as the penalty parameters are increased.}
\label{fig:2shape_JQV}
\end{figure} 
 
The equilibrium shapes and their interior fluid flow shown in Figure \ref{fig:2shape_iteration} reaffirm the classical observation \cite{jaffrin1971peristaltic} of {\em trapping} i.e., an enclosed bolus of fluid particles near the center line indicated by the closed streamlines in the waveframe. As is also well-known, trapping occurs beyond a certain pumping range only; in Fig.~\ref{fig:evolve}, we show the optimal shapes obtained by our algorithm at different flow rates but containing the same volume of fluid. Here, we fixed the bottom wall to be flat. Notice that the bolus appears to form for shapes beyond $Q = 0.9$. Moreover, as expected, the optimal value of power loss is higher for higher flow rates with the extreme case of a flat pipe transporting zero net flow with no power loss.  

\begin{figure}[htbp] \centering
\includegraphics[width=0.6\textwidth]{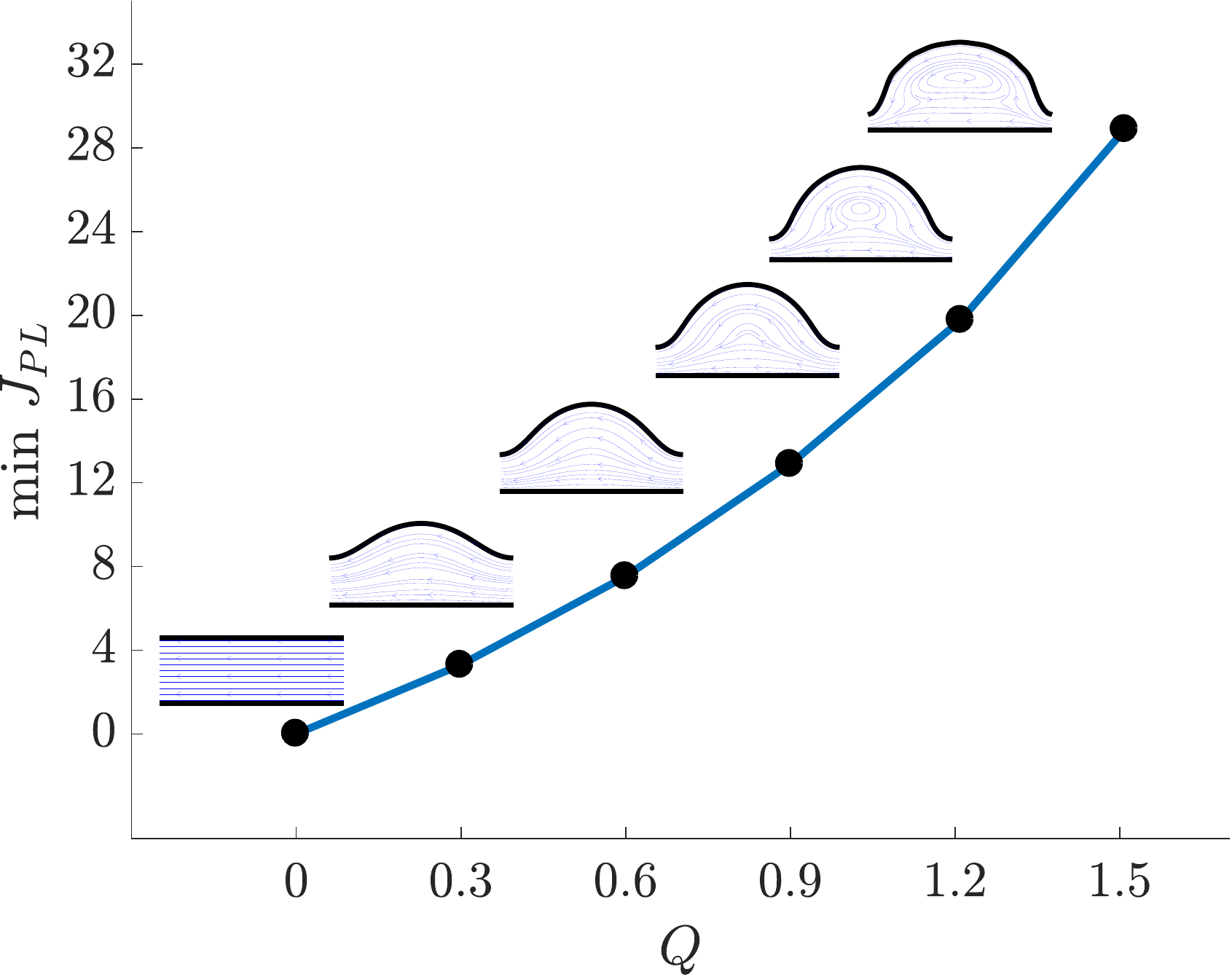}
\caption{Optimal shapes and minimized power losses for varying mass flow rates. The bottom wall is fixed as flat. The constraint of volume is identical for all shapes.}
\label{fig:evolve}
\end{figure}

\section{Conclusions}\label{sc:conclusions}
We derived new analytic formulas for evaluating the shape derivatives of the power loss and the mass flow rate functionals that arise in the shape optimization of Stokesian peristaltic pumps. While we restricted our attention to two-dimensional shapes, extension of these formulas to rotationally-symmetric shapes is rather straightforward. We applied the recently developed periodic boundary integral solver of \cite{barnett:16} to solve the forward/adjoint PDE problems efficiently. 

One of the main directions of our future investigation is to consider optimal shapes for transporting passive (e.g., colloids, bubbles or vesicles) and/or active (e.g., spermatozoa, bacteria) particles in Stokes flow. This is of current interest in both science and technological applications. Clearly, they are more computationally intensive since transient PDEs need to be solved as opposed to the quasi-static ones considered here.  However, particulate flow solvers based on BIEs are proven to be efficient and scalable to simulating millions of deformable particles \cite{petascale}; our work lays the foundation for applying these methods to shape optimization of peristaltic pumps transporting such complex fluids. 

\section{Acknowledgements}
RL and SV acknowledge support from NSF under
grants DMS-1719834 and DMS-1454010. The work of SV was also supported by the Flatiron Institute (USA), a division of Simons Foundation, and by the {\em Fondation Math\'{e}matique Jacques Hadamard} (France).

\appendix
\section{Proofs}
\subsection{Proof of Lemma~\ref{lm4}}
\label{app}

We use the Frenet formulas~\eqref{frenet} and associated conventions. To evaluate $\dd{\bfu}{}\Dsup$, we let $\G$ depend on the fictitious time $\eta$, setting
\begin{equation}
\G_{\eta}\ni\bfx_{\eta}(s) = \bfx(s)+\eta\bfth(s) \qquad(0\shleq s\shleq \ell), \label{wall:shape:eta:2D}
\end{equation}
(where $\G$ stands for $\G^+$ or $\G^-$, and likewise for $\ell$) and seek the relevant derivatives w.r.t. $\eta$ evaluated at $\eta=0$. Note that for $\eta\not=0$, $s$ is no longer the arclength coordinate along $\G_{\eta}$, and $\partial_s\bfx_{\eta}$ is no longer of unit norm; moreover, the length of $\G_{\eta}$ depends on $\eta$. The wall velocity $\bfU=(c\ell/L)\bftau$ for varying $\eta$ is then given by
\begin{equation}
\bfU_{\eta}(s) = \frac{c\ell_{\eta}}{Lg_{\eta}}\partial_s\bfx_{\eta}  \qquad(0\shleq s\shleq \ell),
\label{U:def:eta:2D}
\end{equation}
having set $g_{\eta}=|\partial_s\bfx_{\eta}|$ (note that $g_0=1$). Our task is to evaluate $\text{d}/\text{d}\eta\,\bfU_{\eta}(s)$ at $\eta=0$. We begin by observing that the derivative of $g$ is (since $\partial_s\bfx_{\eta}=\bftau$ and $g=1$ for $\eta=0$)
\[
  \partial_{\eta}g
   = (\partial_s\bfx_{\eta}\sip\partial_{\eta s}\bfx_{\eta})/g = \bftau\sip\partial_s\bfth
   = \partial_s\theta_s - \kappa\theta_n
\]
and the length $\ell_{\eta}$ of $\G_{\eta}$ and its derivative $\dd{\ell}$ are given (noting that $s$ spans the fixed interval $[0,\ell]$ for all curves $\G_{\eta}$) by
\begin{equation}
  \text{(i) \ } \ell_{\eta} = \int_{0}^{\ell} g_{\eta} \ds, \qquad
  \text{(ii) \ } \dd{\ell} = \int_{0}^{\ell} (\partial_s\theta_s - \kappa\theta_n) \ds
 = -\int_{0}^{\ell} \kappa\theta_n \ds.
\label{dd:ell:2D}
\end{equation}
The last equality in (ii), which results from the assumed periodicity of $\bfth$, proves item (b) of the lemma.

Using these identities in~\eqref{U:def:eta:2D}, the sought derivative $\dd{\bfu}{}\Dsup$ of the wall velocity is found as
\begin{equation}
 \dd{\bfu}{}\Dsup = \partial_{\eta}\bfU_{\eta}(s)\Big|_{\eta=0}
 = \frac{c}{L} \lpar \dd{\ell} - \ell(\partial_s\theta_s - \kappa\theta_n) \rpar\bftau
 + \frac{c}{L}\ell\partial_{s}\bfth
 = \frac{c\dd{\ell}}{L}\bftau + \frac{c\ell}{L}(\partial_s\theta_n + \kappa\theta_s)\bfn, \label{dd:uD:2D}
\end{equation}
thus establishing item (a) of the lemma. The proof of Lemma~\ref{lm4} is complete.

\subsection{Proof of Lemma~\ref{lm5}}
\label{lm5:proof}

The proof proceeds by verification, and rests on evaluating $\Div\bfA$, with the vector function $\bfA$ defined by
\[
  \bfA := (\bfsig[\bfu,p]\dip\bfD[\bfuh])\bfth - \bfsig[\bfu,p]\sip\bfna\bfuh\sip\bfth - \bfsig[\bfuh,\hatp]\sip\bfna\bfu\sip\bfth
\]
First, it is easy to check, e.g. using component notation relative to a Cartesian frame, that
\begin{subequations}
\begin{align}
  \Div\lsqb(\bfsig[\bfu,p]\dip\bfD[\bfuh])\bfth\rsqb
 &= (\bfna\bfsig[\bfu,p]\sip\bfth)\dip\bfD[\bfuh] + \bfsig[\bfu,p]\dip(\bfna\bfD[\bfuh]\sip\bfth)
  + (\bfsig[\bfu,p]\dip\bfD[\bfuh])\Div\bfth \label{aux1b} \\
  \Div\lsqb\bfsig[\bfu,p]\sip\bfna\bfuh\sip\bfth\rsqb
 &= (\Div\bfsig[\bfu,p])\sip(\bfna\bfuh\sip\bfth) + \bfsig[\bfu,p]\dip(\bfna\bfuh\sip\bfna\bfth)
   + \bfsig[\bfu,p]\dip(\bfna\bfD[\bfuh]\sip\bfth) \label{aux2b}
\end{align}
\end{subequations}
Next, invoking the constitutive relation~(\ref{forward:pde}b) for both states, one has
\begin{align}
  \bfsig[\bfu,p]\dip(\bfna\bfD[\bfuh]\sip\bfth)
 &= -p\bfna(\Div\bfuh)\sip\bfth + 2\mu\bfD[\bfu]\dip(\bfna\bfD[\bfuh]\sip\bfth) \\
  (\bfna\bfsig[\bfuh,\hatp]\sip\bfth)\dip\bfD[\bfu]
 &= -(\bfna\hatp\sip\bfth)\Div\bfu + 2\mu(\bfna\bfD[\bfuh]\sip\bfth)\dip\bfD[\bfu].
\end{align}
i.e. (using incompressibility)
\begin{equation}
  \bfsig[\bfu,p]\dip(\bfna\bfD[\bfuh]\sip\bfth) = (\bfna\bfsig[\bfuh,\hatp]\sip\bfth)\dip\bfD[\bfu]. \label{aux3b}
\end{equation}
Finally, using~\eqref{aux3b} in~\eqref{aux1b} and the balance equation in~\eqref{aux2b}, together with the corresponding identities obtained by switching $(\bfu,p)$ and $(\bfuh,\hatp)$, one obtains
\begin{align}
  \Div\bfA &= \Div\lsqb (\bfsig[\bfu,p]\dip\bfD[\bfuh])\bfth - \bfsig[\bfu,p]\sip\bfna\bfuh\sip\bfth - \bfsig[\bfuh,\hatp]\sip\bfna\bfu\sip\bfth \rsqb \\
 &= \Div\lsqb \tdemi (\bfsig[\bfu,p]\dip\bfD[\bfuh])\bfth + \tdemi (\bfsig[\bfuh,\hatp]\dip\bfD[\bfu])\bfth
    - \bfsig[\bfu,p]\sip\bfna\bfuh\sip\bfth - \bfsig[\bfuh,\hatp]\sip\bfna\bfu\sip\bfth \rsqb \\
 &= \tdemi \lsqb \bfsig[\bfu,p]\dip\bfD[\bfuh] + \bfsig[\bfuh,\hatp]\dip\bfD[\bfu] \rsqb \Div\bfth
  - \bfsig[\bfu,p]\dip(\bfna\bfuh\sip\bfna\bfth) - \bfsig[\bfuh,\hatp]\dip(\bfna\bfu\sip\bfna\bfth) \\
 &= 2\mu(\bfD[\bfu]\dip\bfD[\bfuh])\Div\bfth - 2\mu\bfD[\bfu]\dip(\bfna\bfuh\sip\bfna\bfth) - 2\mu\bfD[\bfuh]\dip(\bfna\bfu\sip\bfna\bfth)
   + (p\bfna\bfuh + \hatp\bfna\bfu)\dip\bfna\bfth
\end{align}
(with the second equality stemming from $\bfsig[\bfu,p]\dip\bfD[\bfuh]=2\mu\bfD[\bfu]\dip\bfD[\bfuh]=2\mu\bfD[\bfuh]\dip\bfD[\bfu]=\bfsig[\bfuh,\hatp]\dip\bfD[\bfu]$). Using definitions~(\ref{a1:def},b) of $a^1$ and $b^1$, we therefore observe that
\[
  a^1(\bfu,\bfuh,\bfth) - b^1(\bfu,\hatp,\bfth) - b^1(\bfuh,p,\bfth) = \iO \Div\bfA \dV.
\]
The last step consists of applying the first Green identity (divergence theorem) to the above integral. The Lemma follows, with the contribution of the end section $\G_L$ therein stemming from condition (ii) in~\eqref{Theta:def} and the periodicity conditions at the end sections. The latter hold by assumption for both $\bfu$ and $\bfuh$, and the interior regularity of solutions in the whole channel then implies the same periodicity for $\bfna\bfu$ and $\bfna\bfuh$; moreover, periodicity is also assumed for $p$ (but not necessarily for $\hatp$) as well as for $\bfth$.

\subsection{Proof of Lemmas~\ref{forward:wall} and~\ref{adjoint:wall}}
\label{proof:lemmas}

Let points $\bfx$ in a tubular neighborhood $V$ of $\G$ be given in terms of curvilinear coordinates $(s,z)$, so that
\begin{equation}
  \bfx = \bfx(s) + z\bfn(s),
\end{equation}
and let $\bfv(\bfx)=v_s(s,z)\bftau(s)\shp v_n(s,z)\bfn(s)$ denote a generic vector field in $V$. Then, at any point $\bfx\sheq\bfx(s)$ of $\G$, we have
\begin{align}
  \bfna\bfv
 &= \lpar \partial_s v_s \shm \kappa v_n \rpar\bftau\tens\bftau + \lpar \partial_s v_n \shp \kappa v_s \rpar\bfn\tens\bftau
   + \partial_n v_s\bftau\tens\bfn + \partial_n v_n\bfn\tens\bfn, \\
  \Div\bfv
 &= \partial_s v_s - \kappa v_n + \partial_n v_n.
\end{align}
Assuming incompressibility, the condition $\Div\bfv\sheq0$ can be used for eliminating $\partial_n v_n$ and we obtain
\begin{equation}
  \bfna\bfv
 = \lpar \partial_s v_s \shm \kappa v_n \rpar \lpar \bftau\tens\bftau \shm \bfn\tens\bfn \rpar + \lpar \partial_s v_n \shp \kappa v_s \rpar\bfn\tens\bftau + \partial_n v_s\bftau\tens\bfn, \label{D[u]:inc}
\end{equation}
Recalling now that the forward and adjoint solutions respectively satisfy $\bfu=(c\ell/L)\bftau$ and $\bfuh\sheq\bfze$ on $\G$, and that $2\bfD[\bfv]=\bfna\bfv+\bfna\bfv\Tsup$, we have
\begin{equation}
\begin{aligned}
  \bfna\bfu &= \frac{\kappa c\ell}{L}\bfn\tens\bftau + \partial_n u_s\bftau\tens\bfn \quad&
  2\bfD[\bfu] &= \Lpar \frac{\kappa c\ell}{L} + \partial_n u_s \Rpar \lpar \bfn\tens\bftau \shp \bftau\tens\bfn \rpar \\
  \bfna\bfuh &= \partial_n \hatu_s\bftau\tens\bfn \quad&
  2\bfD[\bfuh] &= \partial_n \hatu_s \lpar \bfn\tens\bftau \shp \bftau\tens\bfn \rpar
\end{aligned} \ \Biggr\} \qquad \text{on $\G$.}\label{aux6}
\end{equation}
The corresponding stress vectors $\bff=-p\bfn+2\mu\bfD[\bfu]\sip\bfn$ and $\bffh=-\hatp\bfn+2\mu\bfD[\bfuh]\sip\bfn$ on $\G$ are found as
\begin{equation}
  \bff = -p\bfn + f_s\bftau, \quad \bffh = -\hatp\bfn + \fhat_s\bftau \qquad\text{with}\quad
  f_s = \mu\Lpar \frac{\kappa c\ell}{L} + \partial_n u_s \Rpar, \quad
  \fhat_s = \mu\partial_n \hatu_s,
\end{equation}
which in particular prove items (c) of Lemmas~\ref{forward:wall} and~\ref{adjoint:wall}. Finally,  using the above in~\eqref{aux6} establishes the remaining items (a), (b) of both lemmas.

\bibliography{refs_marc}
\bibliographystyle{plain}
\end{document}